\documentclass[10pt,a4paper,american]{amsart}
\usepackage[left=3.5cm, right=3.5cm]{geometry}
\usepackage[english]{babel}
\usepackage{amsthm}
\usepackage{amsbsy}
\usepackage{amstext}
\usepackage{amsfonts}
\usepackage{mathcomp}
\usepackage{dsfont}
\usepackage{latexsym}
\usepackage{amssymb}
\usepackage{esint}
\usepackage[colorlinks,pdfpagelabels,pdfstartview = FitH,bookmarksopen = true,bookmarksnumbered = true,linkcolor = blue,plainpages = false,hypertexnames = false,citecolor = red,pagebackref=true]{hyperref}
\usepackage{cite}
\usepackage{stmaryrd}
\newtheorem{theo}{Theorem}[section]
\newtheorem{cor}[theo]{Corollary}

\newtheorem{lem}[theo]{Lemma}

\theoremstyle{definition}
\newtheorem{defin}[theo]{Definition}
\newtheorem*{lem*}{Lemma}
\newtheorem{rem}[theo]{Remark}

\newtheorem*{cor*}{Corollary}
\newtheorem*{theo*}{Theorem}

\newcommand{\norm}[1]{\lVert#1\rVert}

\DeclareMathOperator*{\esssup}{ess\,sup}

\def\XXint#1#2#3{{\setbox0=\hbox{$#1{#2#3}{\int}$ }
\vcenter{\hbox{$#2#3$ }}\kern-.6\wd0}}

\def\XXiint#1#2#3{{\setbox0=\hbox{$#1{#2#3}{\iint}$ }
\vcenter{\hbox{$#2#3$ }}\kern-.55\wd0}}

\renewcommand{\d}{\:\:\!\!\mathrm{d}}
\newcommand{\N}{\ensuremath{\mathbb{N}}}
\newcommand{\R}{\ensuremath{\mathbb{R}}}

\renewcommand{\b}{\mathfrak{b}}

\renewcommand{\epsilon}{\varepsilon}
\renewcommand{\rho}{\varrho}

\numberwithin{equation}{section}
\allowdisplaybreaks

\begin{document}
\renewcommand{\refname}{References} 
\renewcommand{\abstractname}{Abstract}

\title[Local boundedness of weak solutions to the DSW equations]{Local boundedness of weak solutions to the Diffusive Wave Approximation of the Shallow Water equations}

\date{\today}
\subjclass[2010]{35B65, 35D30, 35K10}
\keywords{local boundedness, doubly nonlinear parabolic equations}
 
\makeatother

\author[T. Singer]{Thomas Singer}
\address{Thomas Singer\\
Department of Mathematics, Aalto University\\
P.~O.~Box 11100, FI-00076 Aalto University, Finland}
\email{thomas.singer@aalto.fi}

\author[M. Vestberg]{Matias Vestberg}
\address{Matias Vestberg\\
Department of Mathematics, Aalto University\\
P.~O.~Box 11100, FI-00076 Aalto University, Finland}
\email{matias.vestberg@aalto.fi}

\begin{abstract}
In this paper we prove that  weak solutions to the Diffusive Wave Approximation of the Shallow Water equations 
$$
\partial_t u - \nabla\cdot ((u-z)^\alpha|\nabla u|^{\gamma-1}\nabla u) = f
$$ 
are locally bounded. Here, $u$ describes the height of the water,  $z$ is a given function that represents the land elevation and $f$ is a source term accounting for evaporation, infiltration or rainfall.
\end{abstract}

\maketitle

\section{Introduction}

In this work we study regularity properties of weak solutions to the diffusive wave approximation of the shallow water equations (DSW). This parabolic PDE is given by
\begin{align}\label{DSWeq}
\partial_t u - \nabla\cdot ((u-z)^\alpha|\nabla u|^{\gamma-1}\nabla u) = f  \quad \text{ in } \Omega_T:=\Omega\times (0,T),
\end{align}
where $\alpha$ and $\gamma$ are known parameters, $z$ is a given function defined on an open bounded set $\Omega \subset \R^n$, and we are looking for solutions $u$ defined in a parabolic space-time cylinder $\Omega_T$ for $T>0$. In applications, $n=2$, but we have investigated the equation for all $n\geq 2$. Physically, $z$ represents the land elevation, and $u$ describes the height of the water, measured with respect to some ground level. The right-hand side $f$ is a source term accounting for evaporation, infiltration or rainfall. Since the water flows on top of the land, it is natural to consider only solutions satisfying $u\geq z$, which also guarantees that the quantity $(u-z)^\alpha$ appearing in the elliptic term is always  well-defined.

The DSW equation arises as an approximation of the two-dimensional shallow water equations when the inertial terms are neglected while retaining the gravitational terms, and the viscosity is modeled by introducing friction slopes in accordance with Manning's formula. See for example \cite{BoOc} and \cite{DiToLa}. A more empirical, concise derivation can be found in \cite{AlSaDa}. The DSW equation has been used to successfully model flow in wetlands and vegetated areas\cite{FeMo}, and dam breaks\cite{HrBeFr,XaKo}.

Regarding the purely mathematical investigation of the DSW equation, very little is known. In the special case $z=0$, i.e.\ ignoring topographic effects, equation \eqref{DSWeq} was studied by Alonso, Santillana and Dawson in \cite{AlSaDa}. They show the existence of suitably defined weak solutions to the Dirichlet problem with zero values on the lateral boundary. In the same article it is shown that these solutions are bounded if the initial data and the source term are bounded, whereas a numerical investigation of the DSW equation in the general case was done by the two latter authors in \cite{SaDa}. The existence of weak solutions of doubly nonlinear parabolic equations was first proven in \cite{IvMk,IvMkJa,IvMkJa2} for a bounded source term $f$, and for a more general right-hand side in \cite{St}. Let us note that these results also cover the flat case $z= 0$ for the DSW equation.

In the case $z\neq 0$ no mathematical theory is developed. In this paper we start closing this gap by showing that every weak solution to the DSW equation \eqref{DSWeq} is locally bounded. It turns out that we have to assume that $f$ and $z$ are smooth enough, which was already conjectured in \cite[Section 5]{AlSaDa}. For the exact assumptions on the functions $f$ and $z$ we refer to Section \ref{weaksolsect}.

For the parameters $\alpha$ and $\gamma$ we always assume that $\alpha>0$ and $\gamma\in (0,1)$. Furthermore, in this work we are only interested in the slow diffusion case
\begin{align}\label{alphagammacond}
\alpha +\gamma>1.
\end{align} 
The nomenclature slow diffusion traces back to the simple case where $z=0$. Here it is well known that perturbations of solutions only propagate with finite speed, see \cite{BoDuMaSc,St}. Let us note that in applications typically the case  $\alpha \in (1,2)$ is of interest, which is always included in our framework.

In the special case where $\alpha=0$ the DSW equation \eqref{DSWeq} reduces to the $p$-Laplacian with $p=\gamma+1$, whereas if $\gamma=1$ and $z=0$ we end up with the porous medium equation.  

Let us now have a brief look at the history of establishing boundedness for weak solutions. We will use the celebrated method of De Giorgi, which was introduced in the context of linear elliptic equations, cf.\ \cite{De}.  The method is based on the idea to first prove suitable energy estimates on level sets and performing an iteration afterwards by also using Sobolev's embedding.  It turned out that this method is also applicable to nonlinear elliptic equations, see \cite{LaUr}, and also to parabolic equations. In the latter setting, regularity results for linear equations are due to  Ladyzhenskaya, Solonnikov \& Ural'tseva, c.f.\ \cite{LaSoUr}, whereas the nonlinear case was treated by  DiBenedetto in \cite{DiBene}. 
 
Let us shortly explain the main difficulties appearing in our proof. We first observe that it is necessary to consider the function $v:=u-z$ instead of $u$ itself. Also from a physical point of view it is most natural to focus on $v$, since the values of $u$ and $z$ depend on some arbitrarily fixed ground level, whereas their difference is invariant. This however leads to the appearance of a new term involving the function $z$ on the right-hand side of our energy estimate, c.f.\ Section \ref{energysect}. Nevertheless, our method is still applicable if $z$ is smooth enough. 

It is also worth mentioning that we assume that a power of $v$ rather than $v$ itself should have a spatial gradient. This phenomena appears already for the porous medium equation, c.f. \cite{BoDuKoSc,Va} and also for doubly nonlinear equations, see \cite{AltLu,FoSoVe,St,Ve}. This leads to a difficulty when using Sobolev's embedding. The powers of the terms arising in the diffusion part are different from the one that are produced by the parabolic part of the equation. Therefore, it is not possible to directly apply the parabolic Sobolev inequality.

Let us finally mention a well known problem regarding parabolic equations. The main idea is to use the solution itself as testing function, even though it may not possess a weak time derivative. This issue can be solved by using a mollification in time, although a rigorous argument is quite delicate in our setting.

The paper is organized as follows. In Section \ref{weaksolsect} we will present a precise definition of weak solutions and our main result, whereas in Section \ref{Preliminaries} we will introduce some notation and auxiliary tools. Section \ref{energysect} is devoted to proving that weak solutions satisfy certain energy estimates, which are finally utilized in Section \ref{degiorgisect} to prove local boundedness.

\medskip
 
\noindent
{\bf Acknowledgments.} T.~Singer has been supported by the DFG-Project  SI 2464/1-1 ``Highly nonlinear evolutionary problems". M.~Vestberg was partially supported by the V\"ais\"al\"a Foundation and both authors want to express their gratitude to the Academy of Finland. Moreover, we would like to thank Kazuhiro Ishige and Juha Kinnunen for drawing our attention to this topic.

\section{Setting and main result}\label{weaksolsect}

In this section we will state our main result and start by motivating the notion of weak solutions. Already in the easiest case $z=0$ it is necessary to assume that $u^\beta$  belongs to some Sobolev space for some $\beta> 0$ to be chosen later, compare with \cite{AlSaDa,St}. However, in \eqref{DSWeq} such terms do not appear directly, which forces us to interpret terms in a different way.

We reformulate the problem in a way that the factor $(u-z)^\alpha$ which causes the degeneracy of the elliptic part is no longer present. We can rewrite the vector field as
\begin{align*}
(u-z)^\alpha |\nabla u|^{\gamma-1} \nabla u=|(u-z)^\frac{\alpha}{\gamma}\nabla u|^{\gamma-1}(u-z)^\frac{\alpha}{\gamma}\nabla u,
\end{align*}
and note that at least formally,
\begin{align*}
(u-z)^\frac{\alpha}{\gamma}\nabla u&=(u-z)^\frac{\alpha}{\gamma}(\nabla u-\nabla z)+(u-z)^\frac{\alpha}{\gamma}\nabla z\\
&=\beta^{-1}\nabla (u-z)^\beta+(u-z)^\frac{\alpha}{\gamma}\nabla z,
\end{align*}
for the choice 
\begin{align}\label{definition:beta}
\beta:=\frac{\alpha+\gamma}{\gamma}.
\end{align}
 Hence, on a formal level, the DSW equation is equivalent to 
\begin{align}\label{reformulated}
\nabla \cdot \big(\beta^{-\gamma}|\nabla v^\beta+\beta v^\frac{\alpha}{\gamma}\nabla z|^{\gamma-1}(\nabla v^\beta +\beta v^\frac{\alpha}{\gamma}\nabla z)\big)-\partial_t v= - f,
\end{align}
where we have denoted $v=u-z$ to simplify the expression. Note also that $\partial_t u=\partial_t v$, since $z$ is independent of time. This formulation suggest that the natural requirement for solutions is that $v^\beta \in L^p(0,T;W^{1,p}(\Omega))$ for some suitable exponent $p$. The existence results for the flat case in \cite{AlSaDa,St} suggest that the natural choice is $p=\gamma+1$. In order not to overburden our notation we will introduce the vector field
\begin{align*}
A(v,\nabla v^{\beta}):=\beta^{-\gamma}|\nabla v^\beta+\beta v^\frac{\alpha}{\gamma}\nabla z|^{\gamma-1}(\nabla v^\beta +\beta v^\frac{\alpha}{\gamma}\nabla z).
\end{align*}
Furthermore, we have to assume some regularity for the functions $z$ and $f$ to ensure that the weak formulation is well defined. Here we propose that
\begin{align}
\label{assumption:data}
 f\in L^{\frac{\gamma+1}{\gamma}}(\Omega_T) \quad \text{ and } \quad z\in W^{1,\beta(\gamma+1)}(\Omega).
\end{align}
  Thus, we are led to the following definition.

\begin{defin}\label{weakdef}
Suppose that $f$ and $z$ satisfy \eqref{assumption:data}. We say that $u\colon\Omega_T\to \R $ is a weak solution to the DSW equation \eqref{DSWeq} if and only if $v:=u-z$ is nonnegative, 
$
v^\beta \in L^{\gamma+1}(0,T;W^{1,\gamma+1}(\Omega))
$ and 
\begin{align}\label{weakform}
&\iint_{\Omega_T} A(v,\nabla v^\beta)\cdot \nabla \varphi- v\partial_t \varphi\d x\d t=\iint_{\Omega_T} f\varphi \d x\d t,
\end{align}
for all $\varphi \in C^\infty_0(\Omega_T)$. Moreover, if the above statement is true,  we say that $v\colon\Omega_T\rightarrow \R$ is a weak solution to \eqref{reformulated}. 
\end{defin}

In the proof of local boundedness, we require stronger assumptions for the given functions $f$ and $z$. This seems natural because such an assumption is already needed for the $p$-Laplacian, which can be seen as a special case of our equation, c.f.\ \cite{DiBene}. Here, we assume that there exists some $\sigma >\frac{\gamma+1+n}{\gamma+1}$ such that
\begin{align}
\label{assumption:data:regularity}
f\in L^{\frac{\gamma+1}{\gamma}\sigma}(\Omega_T) \quad\text{ and }\quad z\in W^{1,\beta(\gamma+1)\sigma}(\Omega)
\end{align}
holds true. Note that $\beta(\gamma+1)\sigma>n$, which implies that $z$ is H\"older continuous, and in particular, that $z$ is bounded. From the point of view of applications, our assumptions on $f$ and $z$ are not restrictive at all. In realistic models of shallow water flow, $z$ is typically a Lipschitz function, and $f$ is bounded.

Before stating our main result, we will introduce space time cylinders that fit in our setting. For $z_o=(x_o,t_o)\in \Omega_T$ we define 
$$
Q_{\rho}(z_o):=B_\rho(x_o)\times \big(t_o-\rho^{\frac{\beta+1}{\beta}},t_o+\rho^{\frac{\beta+1}{\beta}}\big).
$$
Finally, we present the main result of our paper:

\begin{theo}\label{mainthm}
Let $u$ be a weak solution to the DSW equation \eqref{DSWeq} in the sense of Definition \ref{weakdef}, and suppose that \eqref{assumption:data:regularity} holds true. Then $u$ is locally bounded, and moreover if  $Q_{\rho}(z_o)\Subset\Omega_T$ for some $0<\rho\leq 1$ we have the quantitative estimate
\begin{align*}
\sup_{Q_{\frac\rho 2}(z_o)} (u-z) \leq c\rho^{-\frac{n+\gamma+1}{\beta+1}}\left[1+\iint_{Q_{\rho}(z_o)} (u-z)^{\beta(\gamma+1)} \d x\d t\right]^{\frac{1}{\beta+1}}
\end{align*}
for a constant $c\geq 1$ that depends only on $n$, $\alpha$, $\gamma$, $\|f\|_{L^\sigma(\Omega_T)}$, $\|z\|_{W^{1,\beta(\gamma+1)\sigma }(\Omega)}$ and $\sigma$. 
\end{theo}

We conclude this section by giving a remark on the assumptions for $v$, $z$ and $f$.
\begin{rem}
Since our boundedness result is local it would be sufficient to assume that $v$, $z$ and $f$ belong only to local  Sobolev- and Lebesgue-spaces. 
\end{rem}

\section{Preliminaries}
\label{Preliminaries}

Here we will introduce some notation and present auxiliary tools that will be helpful in the course of the paper.

\subsection{Notation}
 With $B_\rho(x_o)$ we denote the open ball in $\R^n$ with radius $\rho$ and center $x_o$. Furthermore, for$z_o:=(x_o,t_o)\in \Omega_T$  we define a space-time cylinder 
$$Q_{\rho,\theta}(z_o):=B_\rho(x_o)\times (t_o-\theta,t_o+\theta).$$ 
For $v,w \geq 0$ we define 
\begin{align}
\label{definition:F}
\b[v,w]:= \tfrac{\beta}{\beta+1} \big(w^{\beta+1}-v^{\beta+1}\big) +v\big(v^\beta-w^\beta\big),
\end{align}
where $\beta$ is as in \eqref{definition:beta}.  For convenience we will sometimes use the short hand notation $v(\cdot,t)=v(t)$ for $t\in [0,T]$.

\subsection{Auxiliary tools}

We now recall some elementary lemmas that will be used later, and start by defining a mollification in time as it was done in \cite{KiLi}, see also \cite{BoeDuMa}. For $T>0$, $t\in [0,T]$, $h\in (0,T)$ and $w\in L^1(\Omega_T)$   we set
\begin{align}
\label{def:moll}
w_h(x,t):=\frac{1}{h}\int^t_0 e^\frac{s-t}{h}w(x,s)\d s.
\end{align}
Moreover, we define the reversed analogue by
\begin{align*}
w_{\overline h}(x,t) :=\frac{1}{h}\int^T_t e^\frac{t-s}{h}w(x,s)\d s.
\end{align*}
For details regarding the properties of the exponential mollification we refer to \cite[Lemma 2.2]{KiLi},  \cite[Lemma 2.2]{BoeDuMa}, \cite[Lemma 2.9]{St}. The properties of the mollification that we will use have been collected for convenience into the following lemma:

\begin{lem}
\label{expmolproperties} Suppose that $w \in L^1(\Omega_T)$. Then the mollification $w_h$ defined in \eqref{def:moll} has the following properties:
\begin{enumerate}
\item[(i)] 
 If $w\in L^p(\Omega_T)$ then $w_h\in L^p(\Omega_T)$, 
$$
\norm{w_h}_{L^p(\Omega_T)}\leq \norm{w}_{L^p(\Omega_T)},
$$
 and $w_h\to w$ in $L^p(\Omega_T)$.
\item[(ii)]
In the above situation, $w_h$ has a weak time derivative $\partial_t w_h$ on $\Omega_T$ given by
\begin{align*}
\partial_t w_h=\tfrac{1}{h}(w-w_h),
\end{align*}
whereas for $w_{\overline h}$ we have
\begin{align*}
\partial_t w_{\overline h}=\tfrac{1}{h}(w_{\overline h}-w).
\end{align*}
\item[(iii)] 
If $w\in L^p(0,T;W^{1,p}(\Omega))$ then $w_h\to w$ in $L^p(0,T;W^{1,p}(\Omega))$ as $h\to 0$.
\item[(iv)] If $w\in L^p(0,T;L^{p}(\Omega))$ then $w_h \in C^0([0,T];L^{p}(\Omega))$.
\end{enumerate}
\end{lem}

Note that (iv) follows from (i) and (ii) and \cite[Ch.~5.9, Theorem 2]{Ev}.

The next Lemma provides us with some useful estimates for the quantity $\b[w,v]$ that was defined in \eqref{definition:F}. The proof can be found in \cite[Lemma 2.3]{BoDuKoSc}.

\begin{lem}
\label{estimates:boundary_terms}
Let $w,v \geq 0$ and $\beta>0$. Then there exists a constant $c$ depending only on $\beta$ such that
\begin{align*}
\tfrac 1 c\big| v^{\frac{\beta+1}{2}}-w^{\frac{\beta+1}{2}} \big|^2 \leq \b[w,v] \leq c  \big| v^{\frac{\beta+1}{2}}-w^{\frac{\beta+1}{2}} \big|^2.
\end{align*} 
\end{lem}

The following lemma can be proven using an inductive argument, see for example \cite[Lemma 7.1]{Gi}.

\begin{lem}
\label{fastconvg} Let $(Y_j)^\infty_{j=0}$ be a positive sequence such that 
\begin{equation*}
Y_{j+1}\leq C b^j Y^{1+\delta}_j,
\end{equation*}
where $C, b >1$ and $\delta>0$. If
\begin{equation*}
Y_0\leq C^{-\frac{1}{\delta}}b^{-\frac{1}{\delta^2}},
\end{equation*}
then $(Y_j)$ converges to zero as $j\to\infty$.
\end{lem}

\subsection{Continuity in time}
\label{sec:cont_time}
In this section we will show that for every weak solution $u$ of the DSW equation we have that $v=u-z\in C^0([0,T];L^{\beta+1}_{\mathrm{loc}}(\Omega))$. This method is based on the proof of continuity in time of \cite[Section 3.8]{St}. We are only able to show a local version, which is not surprising since our weak formulation is only of local nature. We start by showing an auxiliary lemma.

\begin{lem}
\label{lem:cont_time_step1}
Assume that $v$ is a weak solution to \eqref{reformulated} and by
\begin{align*}
\mathcal{V}:=\left\{ w^\beta\in L^{\gamma+1}(0,T;W^{1,\gamma+1}(\Omega)): \partial_t w^\beta \in L^{\frac{\beta+1}{\beta}}(\Omega_T) \right\}.
\end{align*}
we  denote the class of admissible test functions. Then, for every $\zeta\in C^\infty_0(\Omega_T,\R_{\geq 0})$  and $w\in \mathcal V$  we have
\begin{align}
\label{eq:time_1}
\notag- \iint_{\Omega_T} \partial_t \zeta \b[v,w] \d x\d t =& \iint_{\Omega_T} \zeta \partial_t w^\beta (w-v) -\zeta f(w^\beta-v^\beta) \d x\d t\\
& \quad+ \iint_{\Omega_T} A(v,\nabla v^\beta) \cdot \nabla[\zeta(w^\beta-v^\beta)] \d x\d t.
\end{align}
\end{lem}

\begin{proof}[Proof]
Let $w\in \mathcal V$ and $\zeta \in C^\infty_0(\Omega_T,\R_{\geq 0})$ and choose 
$$
\varphi =\zeta \left( w^\beta-[v^\beta]_h \right)
$$
as testing function in \eqref{weakform}. Moreover,  Lemma \ref{expmolproperties} (ii) implies 
$$
 \big( [v^\beta]_h^{\frac 1 \beta} -v\big) \partial_t [v^\beta]_h\leq 0,
$$
and we can treat the parabolic part as follows
\begin{align*}
\iint_{\Omega_T} v\partial_t \varphi \d x\d t&=\iint_{\Omega_T} \zeta v \partial_t w^\beta \d x\d t -\iint_{\Omega_T} \zeta [v^\beta]_h^{\frac 1 \beta} \partial_t [v^\beta]_h \d x\d t\\
&\quad+ \iint_{\Omega_T} \zeta \big( [v^\beta]_h^{\frac 1 \beta} -v\big) \partial_t [v^\beta]_h \d x\d t+ \iint_{\Omega_T} \partial_t \zeta v \big( w^\beta-[v^\beta]_h \big) \d x\d t \\
&\leq \iint_{\Omega_T} \zeta v \partial_t w^\beta\d x\d t+ \iint_{\Omega_T} \tfrac{\beta}{\beta+1} \partial_t \zeta [v^\beta]_h^{\frac{\beta+1}{\beta}} \d x\d t \\
&\quad +\iint_{\Omega_T} \partial_t \zeta v \big(w^\beta-[v^\beta]_h \big) \d x\d t \\
&\to \iint_{\Omega_T} \zeta v \partial_t w^\beta \d x\d t+ \iint_{\Omega_T} \partial_t \zeta \big(\tfrac{\beta}{\beta+1} v^{\beta+1}+v(w^\beta-v^\beta) \big) \d x\d t\\
&=\iint_{\Omega_T} \zeta \partial_t w^\beta (v-w) \d x\d t -\iint_{\Omega_T} \partial_t \zeta \b[v,w] \d x \d t,
\end{align*}
in the limit $h\downarrow 0$. This shows ``$\geq$'' in \eqref{eq:time_1}. The reverse inequality can be derived in the same way by taking 
$$
\varphi =\zeta \left( w^\beta-[v^\beta]_{\overline h} \right)
$$
as testing function. 
\end{proof}

This brings us into position to show that $v$ is continuous as map from $[0,T]$ to $L^{\beta+1}(\Omega_T)$. To be more precise, we will show the following lemma:

\begin{lem}
\label{lem:cont_time}
Assume that $v$ is a weak solution to \eqref{reformulated}. Then $v\in C^0([0,T];L^{\beta+1}_{\mathrm{loc}}(\Omega))$.
\end{lem}

\begin{proof}[Proof]
The strategy of the proof is to show that $v$ is the uniform limit of continuous functions in $ C^0([0,T];L^{\beta+1}_{\mathrm{loc}}(\Omega))$. By Lemma \ref{expmolproperties} (iv) we have  
$$[v^\beta]_h\in C^0([0,T];L^{\gamma+1}(\Omega))$$ and since $\frac 1\beta <\gamma$ this implies $[v^\beta]_h^{\frac 1 \beta}\in C^0([0,T];L^{\beta+1}(\Omega))$ as well.

For $z_o=(x_o,t_0)\in \Omega_T$  take $\theta,\rho>0$ such that $Q_{2\rho,2\theta}(z_o)\Subset \Omega_T$ and consider cut-off functions $\eta \in C^\infty_0(B_{2\rho}(x_o);[0,1])$ with $\eta \equiv 1$ in $B_\rho(x_o)$ and $|\nabla\eta|\leq \frac 2\rho$, and $\psi \in C^\infty([0,T];[0,1])$ with $\psi \equiv 1$ on $[t_o-\theta,T]$, $\psi \equiv 0$ on $[0,t_o-2\theta]$ and $|\partial_t\psi|\leq \frac 2\theta$ as well as the function 
$$
\chi_\varepsilon(t):=\left\{
\begin{array}{cl}
1 & \text{for } t\in [0,\tau-\varepsilon], \vspace{1mm}\\
-\frac 1\varepsilon (t-\tau) &\text{for } t\in [\tau-\varepsilon,\tau],\vspace{1mm}\\
0& \text{for }t \in [\tau,T],
\end{array}
\right.
$$
where $\tau\in (t_o-\theta,t_o+\theta)$  and $\varepsilon\in(0,t_o-\theta)$. We now exploit Lemma \ref{lem:cont_time_step1} and chose $\zeta=\eta\psi\chi_\varepsilon$ as cutting-off function and take $w\equiv w_h:=[v^\beta]_h^{\frac 1 \beta} $ as comparison map in \eqref{eq:time_1}, which shows for a.e.\ $\tau \in (t_o-\theta,t_o+\theta)$  in the limit $\varepsilon \downarrow 0$
\begin{align*}
\int_{B_\rho(x_o)\times\{\tau\}}\b[v,w_h] \d x &\leq \iint_{\Omega_\tau} \eta \psi \partial_t w_h^\beta (w_h-v) \d x\d t+\iint_{\Omega_\tau} \partial_t \psi \eta \b[v,w_h] \d x\d t\\
&\quad + \iint_{\Omega_\tau} A(v,\nabla v^\beta) \cdot \nabla[\eta \psi(w_h^\beta-v^\beta)] - \eta \psi f(w_h^\beta-v^\beta) \d x\d t\\
&\leq \iint_{Q_{2\rho,2\theta}(z_o)} |\partial_t \psi| \b[v,w_h]  +|f||w_h^\beta-v^\beta|\d x\d t\\
&\quad + \iint_{Q_{2\rho,2\theta}(z_o)} |A(v,\nabla v^\beta)|\left[|\nabla w_h^\beta-\nabla v^\beta|+|\nabla\eta||w_h^\beta-v^\beta| \right]\d x\d t
\end{align*}
By taking first the supremum over $\tau \in (t_o-\theta,t_o+\theta)$ and passing to the limit $h\downarrow 0$ afterwards we conclude
$$
\lim_{h\downarrow 0} \left(\esssup_{\tau \in(t_o-\theta,t_o+\theta)} \int_{B_\rho(x_o)} \b[v,w_h]\d x \right) =0.
$$
Note that this limiting process can be justified by using H\"older's inequality and then Lemma \ref{expmolproperties} for the terms involving the vector field $A$ and the source term $f$, whereas for the term involving the quantity $\b$ we can make use of the fact that $v\in L^{\beta+1}(\Omega_T)$. Finally using Lemma \ref{estimates:boundary_terms} and the fact that $\beta>1$ we end up with
$$
|v-w_h|^{\beta+1}\leq \big|v^{\frac{\beta+1}{2}}-w_h^{\frac{\beta+1}{2}}\big|^2 \leq c(\beta)\b[v,w_h], 
$$
which completes the proof of the lemma.
\end{proof}

\subsection{Mollified weak formulation}
In this section we derive a mollified version of \eqref{weakform}. The main purpose of this procedure is to justify the use of $v^\beta$ as a testing function. Note that such a step is necessary since weak solutions do not possess a weak time derivative in general.

\begin{lem}
\label{lem:mollified}
 Let  $v$ be a weak solution to the DSW equation \eqref{DSWeq} in the sense of Definition \ref{weakdef}. Then we have
\begin{align}\label{h-averaged-form}
\iint_{\Omega_T} [A(v, \nabla v^\beta)]_h\cdot \nabla \phi+\partial_t v_h \phi\d x\d t-\int_{\Omega}v(0)\phi_{\overline h}(0)\d x= \iint_{\Omega_T} f_h\phi\d x\d t,
\end{align}
for all $\phi \in C^\infty_0(\Omega_T)$.
\end{lem}

\begin{proof}[Proof]
Consider the cut-off function 
\begin{align*}
\eta_\varepsilon (t):=\left\{
\begin{array}{cl}
\frac{t}{\varepsilon} & \text{for } t\in[0,\varepsilon]\vspace{1mm}, \\
1 &  \text{for } t\in (\varepsilon, T].
\end{array}
\right.
\end{align*}
For $\phi \in C^\infty_0(\Omega_T)$ we choose $ \varphi=\phi_{\overline h} \eta_\varepsilon$ as testing function in \eqref{weakform}. Letting $\varepsilon\downarrow 0$ is possible since Lemma \ref{lem:cont_time} implies that zero is a Lebesgue point. Using a change of variables in connection with a Fubini argument, as was done for instance in \cite[(3.1)]{BoDuKoSc} or \cite[(2.12)]{KiLi}, shows the claim. 
\end{proof}

\section{Energy Estimates}\label{energysect}

We now proceed to prove the energy estimates, i.e.\ the Caccioppoli inequalities necessary for performing a De Giorgi iteration. In order to do so,  we will make use of the mollified weak formulation, c.f.\ Lemma \ref{lem:mollified}.

\begin{lem}[Caccioppoli estimate] 
Let $v$ be a weak solution to \eqref{reformulated} in the sense of Definition \ref{weakdef}. Then, for all $k\geq 0$ and $\varphi\in C^\infty_0(\Omega_T;\R_{\geq 0})$,
\begin{align}\label{caciocavallo}
\iint_{\Omega_T}& |\nabla (v^\beta-k^\beta)_+|^{\gamma+1}\varphi^{\gamma+1}\d x\d t +\esssup_{\tau\in (0,T)}\int_\Omega\big(v^\frac{\beta+1}{2}-k^\frac{\beta+1}{2}\big)^2_+\varphi^{\gamma+1}(x,\tau)\d x
\\
\notag &\leq c\iint_{\Omega_T} |\nabla \varphi|^{\gamma+1} (v^\beta-k^\beta)^{\gamma+1}_+ \d x\d t  + c\iint_{\Omega_T} \big(v^{\frac{\beta+1}{2}}-k^{\frac{\beta+1}{2}} \big)^2_+ \varphi^\gamma |\partial_t \varphi| \d x\d t
 \\
\notag &\quad+  c\iint_{\Omega_T} |\nabla z|^{\gamma+1}\, v^{\frac{\alpha}{\gamma}(\gamma+1)}\varphi^{\gamma+1} \chi_{\{v>k\}}\d x\d t +c\iint_{\Omega_T} (v^\beta-k^\beta)_+ \varphi^{\gamma+1} |f| \d x\d t
\end{align}
for a constant $c$ that depends only on $\alpha$ and $\gamma$.
\end{lem}

\begin{proof}[Proof]
Let $k\geq 0$. We use the mollified weak formulation \eqref{h-averaged-form} with 
\begin{align*}
\phi=(v^\beta-k^\beta)_+\eta,
\end{align*}
where $\eta \in C^\infty_0(\Omega_T)$ is nonnegative and start investigating the parabolic part by adding and subtracting a suitable term
\begin{align*}
\iint_{\Omega_T}& \partial_t v_h \phi\d x\d t\\
& =\iint_{\Omega_T} \partial_t v_h \left((v_h)^\beta-k^\beta \right)_+\eta \d x\d t+\iint_{\Omega_T} \partial_t v_h \left[ (v^\beta-k^\beta )_+ - ((v_h)^\beta -k^\beta)_+  \right]\eta \d x\d t  \\
&=: \mathrm{I}_h+ \mathrm{II}_h,
\end{align*}
with the obvious meaning of the terms $\mathrm{I}_h$ and $\mathrm{II}_h$. We  estimate the second term with the help of Lemma \ref{expmolproperties} (iii) by
\begin{align*}
\mathrm{II}_h = \iint_{\Omega_T} \tfrac 1 h \left(v-v_h \right) \left[(v^\beta-k^\beta )_+ - ((v_h)^\beta -k^\beta)_+  \right] \eta \geq 0,
\end{align*}
where we used that the map $s\mapsto g(s):= (s^\beta-k^\beta)_+$ is a monotone increasing function.

Now, we are going to treat the first term. Since $v_h\in L^{\beta+1}(\Omega_T)$, we are able to use the chain rule as follows
\begin{align*}
\mathrm{I}_h &= \iint_{\Omega_T} \eta \partial_t v_h g(v_h) \d x \d t \\
&= \iint_{\Omega_T} \eta \partial_t \left[\int_0^{v_h} g(s) \d s \right] \d x\d t \\
&= -\iint_{\Omega_T} \partial_t \eta \int_0^{v_h} g(s) \d s\d x \d t.
\end{align*}
We are interested in  passing to the limit $h \downarrow 0$ in the last calculation. Therefore, we estimate
\begin{align*}
&\left|\iint_{\Omega_T} \partial_t \eta \int_0^{v_h} g(s) \d s\d x \d t -\iint_{\Omega_T} \partial_t \eta \int_0^{v} g(s) \d s\d x \d t \right| \\
&\hspace{6mm} \leq \iint_{\Omega_T} |\partial_t\eta| \left|\int_v^{v_h} g(s) \d s\right| \d x\d t \\
&\hspace{6mm} \leq \iint_{\Omega_T} |\partial_t \eta| \left|\int_v^{v_h} s^\beta \d s\right| \d x \d t \\
&\hspace{6mm} \leq\iint_{\Omega_T}  |\partial_t \eta| \big(v^{\beta}+v_h^\beta \big) |v-v_h|\d x\d t \to 0,
\end{align*}
as $h\downarrow 0$. This can be seen by first using H\"older's inequality  and then  Lemma \ref{expmolproperties} (i) implies that $v_h \to v$ in $L^{\beta+1}(\Omega_T)$. Noting that 
$$
\int_0^{v} g(s) \d s= \b[v,k]\chi_{\{v>k\}},
$$
and using the estimate for $\mathrm{II}_h$ in \eqref{h-averaged-form} and  passing to the limit $h \downarrow 0$, we obtain
\begin{align*}
\iint_{\Omega_T\cap\{u>k\}} &A(v,\nabla v^{\beta})\cdot \nabla\left[(v^\beta-k^\beta)_+\eta\right] -\b[v,k]\partial_t \eta \d x \d t \\
&\leq \iint_{\Omega_T} (v^\beta-k^\beta)_+ \eta f \d x \d t.
\end{align*}
We used the fact that $[A(v,\nabla v^{\beta})]_h\to A(v,\nabla v^{\beta})$ in $L^{\frac{\gamma+1}{\gamma}}(\Omega_T,\R^n)$, $f_h\to f$ in $L^{\frac{\gamma+1}{\gamma}}(\Omega_T)$ and by dominated convergence we get
$$
\int_{\Omega}v(0)\phi_{\overline h}(0)\d x= \int_\Omega\int_0^T\tfrac 1 h
e^{-\frac sh}v(0)\phi(x,s) \d s\d x \to  0,
$$
as $h\downarrow 0$.
We proceed by investigating the diffusion part of our equation
\begin{align*}
\nabla [(v^\beta-k^\beta)_+\eta]=\eta\chi_{\{v>k\}}\nabla v^\beta+(v^\beta-k^\beta)_+\nabla \eta.
\end{align*}
Using this and the definition of $A(v,\nabla v^\beta)$, we can write our last estimate as
\begin{align*}
&\iint_{\Omega_T \cap\{v>k\}} \beta^{-\gamma}|\nabla v^\beta+\beta v^\frac{\alpha}{\gamma}\nabla z|^{\gamma-1}(\nabla v^\beta +\beta v^\frac{\alpha}{\gamma}\nabla z)\cdot (\nabla v^\beta)\eta- \b[v,k]\partial_t\eta\d x\d t 
 \\
 &\quad \leq - \iint_{\Omega_T\cap\{v>k\}}\beta^{-\gamma}|\nabla v^\beta+\beta v^\frac{\alpha}{\gamma}\nabla z|^{\gamma-1}(\nabla v^\beta +\beta v^\frac{\alpha}{\gamma}\nabla z)\cdot \nabla \eta (v^\beta-k^\beta)_+ \d x\d t 
 \\
 &\qquad + \iint_{\Omega_T} (v^\beta-k^\beta)_+ \eta f \d x\d t.
\end{align*}
We add the term $\beta v^\frac{\alpha}{\gamma}\nabla z$ to the expression $(\nabla v^\beta)$ appearing in the first integral on the left-hand side. This simplifies the first integral but produces a term on the right-hand side as follows:
\begin{align*}
 &\iint_{\Omega_T\cap\{v>k\}} \beta^{-\gamma}|\nabla v^\beta+\beta v^\frac{\alpha}{\gamma}\nabla z|^{\gamma+1}\eta - \b[v,k]\partial_t\eta\d x\d t 
 \\
 &\quad\leq -\iint_{\Omega_T}\beta^{-\gamma}|\nabla v^\beta+\beta v^\frac{\alpha}{\gamma}\nabla z|^{\gamma-1}(\nabla v^\beta +\beta v^\frac{\alpha}{\gamma}\nabla z)\cdot \nabla \eta (v^\beta-k^\beta)_+ \d x\d t 
 \\
&\qquad+ \iint_{\Omega_T} \beta^{1-\gamma}|\nabla v^\beta+\beta v^\frac{\alpha}{\gamma}\nabla z|^{\gamma-1}(\nabla v^\beta +\beta v^\frac{\alpha}{\gamma}\nabla z)\cdot\nabla z\, v^\frac{\alpha}{\gamma}\eta \chi_{\{v>k\}}\d x\d t 
\\
&\qquad+ \iint_{\Omega_T} (v^\beta-k^\beta)_+ \eta f \d x\d t.
\end{align*}
We estimate the second integral on the right-hand side upwards by taking the absolute value, applying Cauchy-Schwarz inequality on the inner product and using Young's inequality with the dual exponents $\frac{\gamma+1}{\gamma}$ and $\gamma+1$. In this way we obtain one term that we can absorb to the left-hand side to conclude
\begin{align*}
 &\iint_{\Omega_T\cap \{v>k\}} c^{-1}|\nabla v^\beta+\beta v^\frac{\alpha}{\gamma}\nabla z|^{\gamma+1}\eta- \b[v,k]\partial_t\eta\d x\d t 
 \\
 &\quad\leq \iint_{\Omega_T}\beta^{-\gamma}|\nabla v^\beta+\beta v^\frac{\alpha}{\gamma}\nabla z|^\gamma |\nabla \eta|(v^\beta-k^\beta)_+ +  (v^\beta-k^\beta)_+ \eta |f| \d x\d t
 \\
 &\quad + c\iint_{\Omega_T} |\nabla z|^{\gamma+1}\, v^{\frac{\alpha}{\gamma}(\gamma+1)}\eta \chi_{\{v>k\}} \d x\d t ,
\end{align*}
for a constant $c$ that depends only on $\alpha$ and $\gamma$.
At this point we  choose $\eta=\varphi^{\gamma+1} \xi_{\tau,\varepsilon}$, where $\varphi\in C^\infty_0(\Omega_T;\R_{\geq 0})$, $\tau \in (0,T)$, $\varepsilon\in (0,T-\tau)$ and 
\begin{align*}
\xi_{\tau,\varepsilon}(t)=\begin{cases}
1, &t\in [0,\tau],
\\
1-\varepsilon^{-1}(t-\tau), &t\in (\tau,\tau+\varepsilon],
\\
0, &t\in ( \tau+\varepsilon,T].
\end{cases}
\end{align*}
Although $\eta$ is not smooth, its use can be justified by an approximation argument. With this choice we obtain
\begin{align*}
&\iint_{\Omega_T\cap\{v>k\}}  \b[v,k]\partial_t\eta\d x\d t \\
&\quad=\iint_{\Omega_T\cap\{v>k\}} \b[v,k]\xi_{\tau,\varepsilon} \partial_t\varphi^{\gamma+1}\d x\d t -\varepsilon^{-1}\int^{\tau+\varepsilon}_{\tau}\hspace{-1.5mm}\int_\Omega \chi_{\{v>k\}}\b[v,k]\varphi^{\gamma+1}\d x\d t.
\end{align*}
 Our goal is to pass to the limit $\varepsilon\downarrow 0$. Since $\b[v,k]$ is integrable, the Lebesgue differentiation theorem implies that
\begin{align*}
\lim_{\varepsilon\to 0} \varepsilon^{-1}\int^{\tau+\varepsilon}_{\tau}\hspace{-1.5mm}\int_\Omega  \chi_{\{v>k\}} \b[v,k]\varphi^{\gamma+1}\d x\d t=\int_{\Omega\cap\{v(\tau)>k\}} (\b[v,k]\varphi^{\gamma+1})(x,\tau)\d x,
\end{align*} 
for almost every $\tau \in (0,T)$. In the other integrals taking $\varepsilon\downarrow 0$ poses no problem and we see that
\begin{align*}
\iint_{\Omega_\tau} &|\nabla v^\beta+\beta v^\frac{\alpha}{\gamma}\nabla z|^{\gamma+1}\varphi^{\gamma+1}\chi_{\{v>k\}}\d x\d t + \int_{\Omega\cap \{v(\tau)>k\}} (\b[v,k]\varphi^{\gamma+1})(\cdot,\tau)\d x
 \\
 &\leq c\iint_{\Omega_T}|\nabla v^\beta+\beta v^\frac{\alpha}{\gamma}\nabla z|^\gamma \varphi^\gamma |\nabla \varphi|(v^\beta-k^\beta)_+ + \chi_{\{v>k\}} \b[v,k] |\partial_t\varphi^{\gamma+1}| \d x\d t 
 \\
&\quad+ c\iint_{\Omega_\tau} |\nabla z|^{\gamma+1}\, v^{\frac{\alpha}{\gamma}(\gamma+1)}\varphi^{\gamma+1} \chi_{\{v>k\}}+(v^\beta-k^\beta)_+ \varphi^{\gamma+1} |f| \d x\d t
\end{align*}
for a.e.\ $\tau\in (0,T)$ and a constant $c=c(\alpha,\gamma)$. The first term on the right-hand side can now be treated using Young's inequality. After this we estimate the right-hand side upwards by replacing $\Omega_\tau$ with $\Omega_T$. Finally, using Lemma \ref{estimates:boundary_terms} to estimate the terms involving the function $\b[v,k]$ we obtain
\begin{align*}
 \iint_{\Omega_\tau} &|\nabla v^\beta+\beta v^\frac{\alpha}{\gamma}\nabla z|^{\gamma+1}\varphi^{\gamma+1}\chi_{\{v>k\}}\d x\d t + \int_{\Omega} \big(\big(v^{\frac{\beta+1}{2}}-k^{\frac{\beta+1}{2}} \big)^2_+\varphi^{\gamma+1}\big)(\tau)\d x
 \\
 &\leq c\iint_{\Omega_T} |\nabla \varphi|^{\gamma+1} (v^\beta-k^\beta)_+^{\gamma+1} +\big(v^{\frac{\beta+1}{2}}-k^{\frac{\beta+1}{2}} \big)^2_+\varphi^\gamma |\partial_t \varphi|\d x\d t 
 \\
&\quad+ c\iint_{\Omega_T} |\nabla z|^{\gamma+1}\, v^{\frac{\alpha}{\gamma}(\gamma+1)}\varphi^{\gamma+1} \chi_{\{v>k\}} + (v^\beta-k^\beta)_+ \varphi^{\gamma+1} |f| \d x\d t,
\end{align*}
for $c=c(\gamma,\alpha)$.
Analyzing each term on the left-hand side separately, we can let $\tau \uparrow T$ in the first term and take the essential supremum in the second one. This leads to 
\begin{align}\label{almost_there}
\iint_{\Omega_T} &|\nabla v^\beta+\beta v^\frac{\alpha}{\gamma}\nabla z|^{\gamma+1}\varphi^{\gamma+1}\chi_{\{v>k\}}\d x\d t \\
\notag&\quad+ \esssup_{\tau\in (0,T)}\int_\Omega \big(\big(v^{\frac{\beta+1}{2}}-k^{\frac{\beta+1}{2}} \big)^2_+\varphi^{\gamma+1}\big)(\tau)\d x
 \\
\notag &\leq c\iint_{\Omega_T} |\nabla \varphi|^{\gamma+1} (v^\beta-k^\beta)^{\gamma+1}_+ +\big(v^{\frac{\beta+1}{2}}-k^{\frac{\beta+1}{2}} \big)^2_+ \varphi^\gamma|\partial_t \varphi| \d x\d t
 \\
\notag &\quad+  c\iint_{\Omega_T} |\nabla z|^{\gamma+1}\, v^{\frac{\alpha}{\gamma}(\gamma+1)}\varphi^{\gamma+1} \chi_{\{v>k\}} + (v^\beta-k^\beta)_+ \varphi^{\gamma+1} |f| \d x\d t,
\end{align}
for a constant $c$ depending only on $\alpha$ and $\gamma$. In order to modify the first term on the left-hand side we note that
\begin{align*}
|\nabla (v^\beta-k^\beta)_+|^{\gamma+1}&=\chi_{\{v>k\}}|\nabla v^\beta|^{\gamma+1}= \chi_{\{v>k\}}|\nabla v^\beta+\beta v^\frac{\alpha}{\gamma}\nabla z-\beta v^\frac{\alpha}{\gamma}\nabla z|^{\gamma+1}
\\ &\leq 2^\gamma \chi_{\{v>k\}} |\nabla v^\beta+\beta v^\frac{\alpha}{\gamma}\nabla z|^{\gamma+1} + 2^\gamma\beta^{\gamma+1} |\nabla z|^{\gamma+1}\, v^{\frac{\alpha}{\gamma}(\gamma+1)}\chi_{\{v>k\}}.
\end{align*}
Looking at this estimate, we see that the first term on the left-hand side of \eqref{almost_there} can be replaced by
\begin{align*}
\iint_{\Omega_T} |\nabla (v^\beta-k^\beta)_+|^{\gamma+1}\varphi^{\gamma+1}\d x\d t.
\end{align*}
This finishes the proof of the Lemma.
\end{proof}

As an immediate consequence we have:

\begin{cor}
\label{cor:energy_est}
Let $v$ be a weak solution to \eqref{reformulated} in the sense of Definition \ref{weakdef}. Then for every $Q_{R,\theta}(z_o)\Subset \Omega_T$, $0<s<\theta$, $0<\rho<R$ and $k\geq 0$ we have
\begin{align}\label{energy_est_cylinder}
\iint_{Q_{\rho,s}(z_o)}& |\nabla (v^\beta-k^\beta)_+|^{\gamma+1}\d x\d t +\esssup_{\tau\in [t_o-s,t_o+s]}\int_{B_\rho(x_o)}\big(v^\frac{\beta+1}{2}(\tau)-k^\frac{\beta+1}{2}\big)^2_+\d x \notag
\\
\notag &\leq c\iint_{Q_{R,\theta}(z_o)} \frac{ (v^\beta-k^\beta)^{\gamma+1}_+}{(R-\rho)^{\gamma+1}} +\frac{\big(v^{\frac{\beta+1}{2}}-k^{\frac{\beta+1}{2}} \big)^2_+}{\theta-s}  \d x\d t
 \\
 &\quad+  c\iint_{Q_{R,\theta\cap\{v>k\}}(z_o)}  v^{\beta(\gamma+1)} + (v^\beta-k^\beta)_+^{\gamma+1}  +|f|^{\frac{\gamma+1}{\gamma}} +|\nabla z|^{\beta(\gamma+1)}\d x\d t 
\end{align}
for a constant $c$ that depends only on $\alpha$ and $\gamma$.
\end{cor}

\begin{proof}[Proof]
Taking $\varphi\in C^\infty_0(Q_{R,\theta}(z_0);[0,1])$ with $\varphi\equiv 1$ on $Q_{r,s}(z_o)$ and 
$$
|\nabla\varphi| \leq \frac{2}{R-\rho}  \quad \text{ and } \quad |\partial_t \varphi|\leq \frac{2}{\theta-s}
$$
as a testing function in \eqref{caciocavallo} and using Young's inequality  implies the claimed estimate. 
\end{proof}

\section{Local boundedness of weak solutions}\label{degiorgisect}

In this section we show that weak solutions the DSW equation \eqref{DSWeq} are locally bounded. 
This will be done by iterating an inequality, which we obtain by using the Caccioppoli estimate from Corollary \ref{cor:energy_est} together with the Sobolev embedding. 

\vspace{3mm}

\begin{proof}[Proof of Theorem \ref{mainthm}]
Set 
$$
m:=\frac{\beta+1}{\beta}=1+\frac 1 \beta=1+\frac{\gamma}{\alpha+\gamma} <1+\gamma
$$
and consider $z_o \in \Omega_T$ and $0<\rho\leq 1$ so small that $Q_{\rho,\rho^m}(z_o) \Subset \Omega_T$. For $j\in \N_0$ we define sequences of radii $\rho_j,\hat \rho_j$ by
$$
\rho_j=\tfrac 12 (1+2^{-j})\rho  \quad \text{ and } \quad \hat \rho_j =\tfrac 12 (\rho_j+\rho_{j+1})
$$
as well as times $\tau_j,\hat \tau_j$ by
$$
\tau_j = \left(\frac{\rho}{2}\right)^m \left(1+\frac{2^m-1}{2^{mj}} \right) \quad \text{ and }\quad \hat \tau_j = \tfrac 12 (\tau_j+\tau_{j+1})
$$ 
and also levels  $k_j$ by
$$
k_j =k \left(1-2^{-j}\right)^{\frac{2}{\beta+1}}
$$
where $k\geq 1$ is a parameter to be chosen later. Moreover, we introduce the cylinders 
$$
Q_j:=Q_{\rho_j,\tau_j}(z_o)  \quad \text{ and } \quad \widehat Q_j :=Q_{\hat\rho_j,\hat\tau_j}(z_o) 
$$
 In the following we consider the sequence of integrals 
\begin{align}
\label{def:sequence}
Y_j := \iint_{Q_j} \Big(v^{\frac{\beta+1}{2}}-k_j^{\frac{\beta+1}{2}} \Big)^{\frac{2\beta(\gamma+1)}{\beta+1}}_+ \d x\d t ,
\end{align}
which are finite since $v\in L^{\beta(\gamma+1)}(\Omega_T)$.
The idea of the proof is to exploit Lemma \ref{fastconvg}. The first step is to show a suitable recursive estimate for $Y_j$. It turns out that our free parameter $k$ can be chosen large enough so that the quantity $Y_0$ is small enough and we obtain $Y_j\to 0$ as $j\to \infty$. This implies 
$$
\iint_{Q_{\frac \rho 2, (\frac{\rho}{2})^m}(z_o)} \big(v^{\frac{\beta+1}{2}}-k^{\frac{\beta+1}{2}} \big)^2_+ \d x\d t =0,
$$
and we obtain $\esssup_{Q_{\frac \rho 2, (\frac{\rho}{2})^m}(z_o)}v \leq k$. Another outcome of the proof is the fact that we can determine $k$ in terms of the $L^{\beta(\gamma+1)}$-norm of $v$.
 
By these observations it remains to show a suitable iterative estimate for $Y_j$. Therefore, we consider a sequence of cut-off functions $\zeta_j\in C^\infty_0(B_{\hat \rho_j}(x_o);[0,1])$ with $\zeta_j \equiv 1$ in $B_{\rho_{j+1}}(x_o)$ and $|\nabla\zeta_j| \leq \frac{16\cdot 2^j}{\rho}$. With this particular choice we obtain
\begin{align}
\label{iteration_start}
Y_{j+1} &\leq \iint_{\widehat Q_j} \Big(\big(v^{\frac{\beta+1}{2}}-k_{j+1}^{\frac{\beta+1}{2}} \big)^{\frac{2\beta}{\beta+1}}_+ \zeta_j \Big)^{\gamma+1} \d x\d t
\\
\notag&\leq \left[\iint_{\widehat Q_j} \Big(\big(v^{\frac{\beta+1}{2}}-k_{j+1}^{\frac{\beta+1}{2}} \big)^{\frac{2\beta}{\beta+1}}_+ \zeta_j \Big)^{(\gamma+1)\frac{n+m}{n}} \d x \d t \right]^{\frac{n}{n+m}}\\
\notag& \qquad\cdot \left|\widehat Q_j\cap\{v>k_{j+1}\}\right|^{1-\frac{n}{n+m}}.
\end{align}
To simplify the notation, we write the function appearing in the first factor as 
$$
\phi:= \big(v^{\frac{\beta+1}{2}}-k_{j+1}^{\frac{\beta+1}{2}} \big)^{\frac{2\beta}{\beta+1}}_+ \zeta_j \leq (v^\beta-k_{j+1}^\beta)_+ \zeta_j,
$$
where we are able to estimate $v$ in this way since $\beta>1$. We proceed by estimating the first term in \eqref{iteration_start} and note that $\gamma+1 <2\leq n$. This allows us to first use H\"older's inequality and then Sobolev's inequality slice-wise to conclude
\begin{align*}
&\iint_{\widehat Q_j} \phi^{(\gamma+1)\frac{n+m}{n}} \d x \d t =\int_{t_o-\hat\tau_j}^{t_o+\hat \tau_j}\int_{B_{\hat \rho_j}(x_o)} \phi^{\gamma+1}\phi^{(\gamma+1)\frac{m}{n}} \d x\d t
\\
&\quad \leq \int_{t_o-\hat\tau_j}^{t_o+\hat \tau_j} \left[\int_{B_{\hat \rho_j}(x_o)}\big((v^\beta-k_{j+1}^\beta)_+\zeta_j \big)^{\frac{n(\gamma+1)}{n-(\gamma+1)}} \d x \right]^{\frac{n-(\gamma+1)}{n}} 
\left[\int_{B_{\hat \rho_j}(x_o)} \phi^{m} \d x\right]^{\frac{\gamma+1}{n}} \d t 
\\
&\quad \leq c\left[\esssup_{t\in(t_o-\hat \tau_j,t_o+\hat \tau_j)}\int_{B_{\hat \rho_j}(x_o)} \phi^{m} \d x \right]^{\frac{\gamma+1}{n}} \int_{t_o-\hat\tau_j}^{t_o+\hat \tau_j} \int_{B_{\hat \rho_j}(x_o)}|\nabla \big( (v^\beta-k_{j+1}^\beta)_+\zeta_j\big) |^{\gamma+1} \d x\d t 
\\
&\quad \leq c\left[\esssup_{t\in(t_o-\hat \tau_j,t_o+\hat \tau_j)}\int_{B_{\hat \rho_j}(x_o)} \big(v^{\frac{\beta+1}{2}}-k_{j+1}^{\frac{\beta+1}{2}} \big)_+^{2} \d x \right]^{\frac{\gamma+1}{n}} 
\\
&\hspace{12mm} \cdot \iint_{\widehat Q_j} |\nabla (v^\beta-k_{j+1}^\beta)_+|^{\gamma+1}\zeta^{\gamma+1}_j+ (v^\beta-k_{j+1}^\beta)_+^{\gamma+1}|\nabla \zeta_j|^{\gamma+1} \d x \d t 
\\
&\quad \leq c\left[\esssup_{t\in(t_o-\hat \tau_j,t_o+\hat \tau_j)}\int_{B_{\hat \rho_j}(x_o)} \big(v^{\frac{\beta+1}{2}}-k_{j+1}^{\frac{\beta+1}{2}} \big)_+^{2} \d x \right]^{\frac{\gamma+1}{n}} \\
&\hspace{12mm} \cdot \iint_{\widehat Q_j} |\nabla (v^\beta-k^\beta_{j+1})_+|^{\gamma+1} + 2^{j(\gamma+1)}\frac{(v^\beta-k_{j+1}^\beta)_+^{\gamma+1}}{\rho^{\gamma+1}} \d x \d t,
\end{align*}
where $c$ depends only on $\alpha$, $\gamma$ and $n$.
Next, we are going to estimate the right-hand side of the previous inequality with the help of the Caccioppoli inequality \eqref{energy_est_cylinder} 
\begin{align*}
\esssup_{t\in(t_o-\hat \tau_j,t_o+\hat \tau_j)} &\int_{B_{\hat \rho_j}(x_o)} \big(v^{\frac{\beta+1}{2}}-k_{j+1}^{\frac{\beta+1}{2}} \big)_+^{2} \d x + \iint_{\widehat Q_j} |\nabla (v^\beta-k^\beta_{j+1})_+|^{\gamma+1}  \d x\d t 
\\
& \leq c \iint_{Q_j}\frac{(v^\beta-k_{j+1}^\beta)^{\gamma+1}_+}{(\rho_j-\hat\rho_j)^{\gamma+1}} +\frac{\big(v^{\frac{\beta+1}{2}}-k_{j+1}^{\frac{\beta+1}{2}} \big)_+^2}{\tau_j-\hat\tau_j} \d x\d t
\\
&\quad +c\iint_{Q_j\cap\{v>k_{j+1}\}}  v^{\beta(\gamma+1)} +|f|^{\frac{\gamma+1}{\gamma}}+ (v^\beta-k_{j+1}^\beta)^{\gamma+1}_+ + |\nabla z|^{\beta(\gamma+1)} \d x\d t 
\\
&\leq c \iint_{Q_j} 2^{j(\gamma+1)}\frac{(v^\beta-k_{j+1}^\beta)^{\gamma+1}_+}{\rho^{\gamma+1}} + 2^{mj}\frac{\big(v^{\frac{\beta+1}{2}}-k_{j+1}^{\frac{\beta+1}{2}} \big)_+^2}{\rho^m} \d x\d t
\\
&\quad +c\iint_{Q_j\cap\{v>k_{j+1}\}} v^{\beta(\gamma+1)} + (v^\beta-k_{j+1}^\beta)^{\gamma+1} + |f|^{\frac{\gamma+1}{\gamma}}+|\nabla z|^{\beta(\gamma+1)} \d x\d t
\\
&\leq c \iint_{Q_j\cap\{v>k_{j+1}\}} 2^{j(\gamma+1)}\left[ \frac{\big(v^{\frac{\beta+1}{2}}-k_{j+1}^{\frac{\beta+1}{2}} \big)_+^{2\frac{\beta}{\beta+1}(\gamma+1)}}{\rho^{\gamma+1}}+\frac{k_{j+1}^{\beta(\gamma+1)}}{\rho^{\gamma+1}}\right] \d x\d t 
\\
&\quad +c\iint_{Q_j\cap\{v>k_{j+1}\}} |f|^{\frac{\gamma+1}{\gamma}}+|\nabla z|^{\beta(\gamma+1)}+1 \d x\d t
\end{align*}
for a constant $c$ depending only on $\alpha$, $\gamma$ and $n$.
Note that in the last line we used Young's inequality, the fact that $0<\rho\leq 1$ and the estimate 
$$
(v^\beta-k_{j+1}^\beta)^{\gamma+1}\leq v^{\beta(\gamma+1)}\leq c(\beta,\gamma) \left[\big(v^{\frac{\beta+1}{2}}-k_{j+1}^{\frac{\beta+1}{2}}\big)^{\frac{2\beta(\gamma+1)}{\beta+1}}+k_{j+1}^{\beta(\gamma+1)}\right].
$$
Using the abbreviation $G:=|f|^{\frac{\gamma+1}{\gamma}}+|\nabla z|^{\beta(\gamma+1)}+1$ and combining the last two estimates shows
\begin{align*}
&\iint_{\widehat Q_j} \phi^{(\gamma+1)\frac{n+m}{n}} \d x \d t \\
&\quad\leq  c \left[\iint_{Q_j\cap\{v>k_{j+1}\}} 2^{j(\gamma+1)}\left[ \frac{\big(v^{\frac{\beta+1}{2}}-k_{j+1}^{\frac{\beta+1}{2}} \big)_+^{2\frac{\beta}{\beta+1}(\gamma+1)}}{\rho^{\gamma+1}}+\frac{k_{j+1}^{\beta(\gamma+1)}}{\rho^{\gamma+1}}\right]+G \d x\d t  \right]^{1+\frac{\gamma+1}{n}}.
\end{align*}
Let us notice that
\begin{align*}
\big| Q_j&\cap\{v>k_{j+1}\}\big| \big( k_{j+1}^{\frac{\beta+1}{2}}-k_j^{\frac{\beta+1}{2}}\big)^{2\beta\frac{\gamma+1}{\beta+1}} \\
&\leq \iint_{ Q_j\cap\{v>k_{j+1}\}} \big(v^{\frac{\beta+1}{2}}-k_j^{\frac{\beta+1}{2}}\big)^{2\beta\frac{\gamma+1}{\beta+1}}_+ \d x\d t \leq Y_j
\end{align*}
and since $k\geq 1$ this implies
$$
\big| Q_j\cap\{v>k_{j+1}\}\big| \leq \frac{2^{(j+1)\frac{2\beta(\gamma+1)}{\beta+1}}}{k^{\beta(\gamma+1)}}Y_j \leq 2^{(j+1)\frac{2\beta(\gamma+1)}{\beta+1}} Y_j.
$$ 
This finally allows us to estimate the function $\phi^{(\gamma+1)\frac{n+m}{n}}$ by
\begin{align*}
\iint_{\widehat Q_j}& \phi^{(\gamma+1)\frac{n+m}{n}} \d x \d t \\
&\leq c\left[\frac{2^{j\frac{2\beta(\gamma+1)}{\beta+1}}}{\rho^{\gamma+1}}Y_j+\|G\|_{L^\sigma(\Omega_T)} \big| Q_j\cap\{v>k_{j+1}\}\big|^{1-\frac1 \sigma}\right]^{1+\frac{\gamma+1}{n}} \\
&\leq c \left[\frac{2^{j\frac{2\beta(\gamma+1)}{\beta+1}}}{\rho^{\gamma+1}}\left(1+\|v\|_{L^{\beta(\gamma+1)}(Q_\rho(z_o))}\right)^{\frac{\beta(\gamma+1)}{\sigma}}Y_j^{1 -\frac 1 \sigma} \right]^{1+\frac{\gamma+1}{n}},
\end{align*}
for a constant $c=c(\alpha,\gamma,n,\|G\|_{L^\sigma}(\Omega_T))$ and where we also used that $Y_j \leq \|v\|^{\beta(\gamma+1)}_{L^{\beta(\gamma+1)}(Q_{\rho}(z_o))}$.
Utilizing \eqref{iteration_start} and the last estimate shows
\begin{align*}
Y_{j+1} &\leq c\left[\frac{2^{j\frac{2\beta(\gamma+1)}{\beta+1}}}{\rho^{\gamma+1}} \left(1+\|v\|_{L^{\beta(\gamma+1)}(Q_\rho(z_o))}\right)^{\frac{\beta(\gamma+1)}{\sigma}}Y_j^{1 -\frac 1 \sigma} \right]^{\frac{n+\gamma+1}{n+m}}
\cdot \left[ \frac{2^{j\frac{2\beta(\gamma+1)}{\beta+1}}}{k^{\beta(\gamma+1)}} Y_j\right]^{\frac{m}{n+m}} \\                             
& \leq c \left(1+\|v\|_{L^{\beta(\gamma+1)}(Q_\rho(z_o))}\right)^{\frac{\beta(\gamma+1)(n+\gamma+1)}{(n+m)\sigma}} \frac{2^{j\frac{2\beta(\gamma+1)}{\beta+1}\frac{m+n+\gamma+1}{m+n}}}{\rho^{(\gamma+1)\frac{n+\gamma+1}{n+m}}} \frac{Y_j^{1+\frac{1}{n+m}(\gamma+1-\frac 1 \sigma(n+\gamma+1))}}{k^{\beta(\gamma+1)\frac{m}{n+m}}} \\
&= Cb^jY_j^{1+\delta}
\end{align*}
where we used the notation
\begin{align*}
&C=\frac{c}{k^{\beta(\gamma+1)\frac{m}{n+m}}\rho^{(\gamma+1)\frac{n+\gamma+1}{n+m}}}\left(1+\|v\|_{L^{\beta(\gamma+1)}(Q_\rho(z_o))}\right)^{\frac{\beta(\gamma+1)(n+\gamma+1)}{(n+m)\sigma}} ,\\
& b=2^{\frac{2\beta(\gamma+1)}{\beta+1}\frac{m+n+\gamma+1}{m+n}}, \\
& \delta=\tfrac{1}{n+m}\big(\gamma+1-\tfrac 1 \sigma(n+\gamma+1)\big)
\end{align*}
for some constant $c$ depending only on $n$, $\alpha$, $\gamma$ and $\|G\|_{L^\sigma(\Omega_T)}$. Note that $\delta>0$ holds true by the assumption $\sigma>\frac{n+\gamma+1}{\gamma+1}$. In order to apply Lemma \ref{fastconvg} we have to ensure that 
$$
Y_0 =\iint_{Q_{\rho,\rho^m}(z_o)} v^{\beta(\gamma+1)} \d x \d t\leq C^{-\frac 1 \delta}b^{-\frac{1}{\delta^2}}
$$
is satisfied. This can be realized by choosing $k\geq 1$ large enough. For instance, if we take
$$
k= c\rho^{-\frac{n+\gamma+1}{\beta m}}\left[1+\iint_{Q_{\rho,\rho^m}(z_o)} v^{\beta(\gamma+1)} \d x\d t\right]^{\frac{1}{\beta m}}
$$ 
with a suitable constant $c$ depending only on $\alpha$, $\gamma$, $n$, $\|G\|_{L^\sigma(\Omega_T)}$ and $\sigma$, all requirements of Lemma \ref{fastconvg} are satisfied and $Y_j\to 0$ as $j\to \infty$ which implies
$$
\sup_{Q_{\frac \rho 2, (\frac \rho 2)^m}(z_o)} v \leq k.
$$
This finishes the proof of Theorem \ref{mainthm}.
\end{proof}

%

\end{document}